\documentclass[12pt,oneside]{amsart}

\usepackage{amssymb,amsmath}
\usepackage{hyperref}
\usepackage{amsthm}
\usepackage{footnote}
\usepackage{textcomp}
\usepackage[english]{babel}
\usepackage{stackengine}
\usepackage{scalerel}
\usepackage{graphicx, color}
\numberwithin{equation}{section}
\newtheorem{theorem}{Theorem}[section]
\newtheorem{lemma}[theorem]{Lemma}
\newtheorem{proposition}[theorem]{Proposition}
\newtheorem{corollary}[theorem]{Corollary}
\newtheorem{conjecture}[theorem]{Conjecture}

\theoremstyle{definition}
\newtheorem{definition}[theorem]{Definition}

\newtheorem{remark}[theorem]{Remark}

\numberwithin{equation}{section}

\begin{document}

%%%%%%%%%%%%%%%%%%%%%%%%%%%%%%%%%%%%%%%%%%%%%%%%%%%%%%%%%%%%%%%%%%%%%%%%

\title[Apolar Star Configurations]{Special apolar subset: the case of star configurations }
\thanks{The authors acknowledge that the present research has been partially supported by MIUR grant Dipartimento di Eccellenza 2018-2022 (E11G18000350001)}

\author{Iman Bahmani Jafarloo}
\address{DISMA, Dipartimento di Scienze Matematiche, Politecnico di Torino, Corso Duca degli Abruzzi 24, 10129 Turin, Italy. \newline \hspace*{0.3cm}  Dipartimento di Matematica, Universit\`a degli Studi di Torino, Via Carlo Alberto 10, 10123 Turin, Italy.}
\email{iman.bahmanijafarloo@polito.it,iman.bahmanijafarloo@unito.it}

\author{Enrico Carlini}
\address{DISMA, Dipartimento di Scienze Matematiche, Politecnico di Torino, Corso Duca degli Abruzzi 24, 10129 Turin, Italy.}
\email{enrico.carlini@poito.it}


\begin{abstract}
In this paper we consider a generic  degree $d$  form $ F $ in $n+1$ variables. In particular, we investigate the existence of star configurations apolar to $F$, that is the existence of apolar sets of points obtained by the $ n $-wise intersection of $ r $ general hyperplanes of $ \mathbb{P}^n $. We present a complete answer for all values of $(d,r,n)$ except for $(d,d+1,2)$ when we present an algorithmic approach.
\end{abstract}

\maketitle

%%%%%%%%%%%%%%%%%%%%%%%%%%%%%%%%%%%%%%%%%%%%%%%%%%%%%%%%%%%%%%%%%%%%%%%%%%%%%%%

\section{Introduction}

The study of the decomposition of tensors as the sum of simpler tensors has attracted a huge amount of research from pure and applied mathematics, see for example \cite{paperone} and the references therein. In spite of the many efforts many basic questions stay open, even for special family of tensors, such as symmetric tensors. Symmetric tensors are of particular interest since they correspond to homogeneous polynomials or forms. Writing a degree $d$ form $F$ as the sum of $d$-th powers is the so called Waring problem for forms, see \cite{geramita1996inverse}, namely one wants to find a sum of powers decomposition of $F$, that is an expression of the form
\[
F=L_1^d+\cdots+L_s^d,
\]
where the forms $ L_i $ have degree one.

One of the most interesting quantities related to sum of power decomposition is the (Waring) rank of a form $F$, denoted as $\mathrm{rk}(F)$, which is defined as the minimal number of linear forms need to write down $F$ as a sum of powers. We note that, in spite of the numerous efforts, the rank is only known for special family of forms, for examples: quadrics, that is $d=2$, binary forms, that is forms in two variables, cubics in three variables, and monomials, see \cite{carlini2012thesolution}. Moreover, the value of the rank is known for sufficiently general forms, sometimes called generic forms, see \cite{alexander1995polynomial}.

Using the Apolarity Lemma, see Lemma \ref{ourapolaritylemma}, one can study sum of powers decompositions of $F$ by studying sets of points apolar  to $F$, that is set of points $\mathbb{X}$ having the defining ideal $I(\mathbb{X})$ contained in the apolar ideal of $F$, see Subsection \ref{sectionbasicfacts}. The Apolarity Lemma allows us to give a geometrical flavor to the Waring problem, for example,  the rank of $F$ is just the minimal degree of a smooth apolar subset to $F$.

Very little is known on the geometry of apolar subsets in general. However, there are cases for which we know quite a lot. This is the case, for example, of monomials and of cusps. For monomials, we know that the minimal apolar subsets are necessarily complete intersections: see \cite{BUCZYNSKA201345} for a proof and \cite{carlini2017real} for an interesting application related to the real Waring rank. For cusps, see \cite{carlini2017waring}, all minimal apolar subsets split as a single point union a set of degenerate points, that is points lying on a hyperplane.

In this paper we investigate the connection between apolar subsets and star configuration sets of points. A star configuration set of points $\mathbb{X}(r)$, see \cite{GERAMITA2013279,starongenerichyper}, is a set of points in $\mathbb{P}^n$ obtained as the $n$-wise intersection of $r$ hyperplanes in general position. The interest in star configuration set of points is well established for two different reasons. On the one hand star configurations are general enough, for example with respect to the Hilbert function, see Theorem \ref{theoremin"n"}. On the other hand, star configurations are very special, for example with respect to their ideals, see again Theorem \ref{theoremin"n"}. This mix of generality and speciality makes star configurations of special interest.

In particular, we consider the following question: does the generic degree $d$ form in $n+1$ variables have an apolar star configuration $\mathbb{X}(r)$?
We give a complete answer to this question for all 3-ples $(d,r,n)$ except for the family $(d,d+1,2)$ for which we only present some special results for some values of $d$.

The paper is structured as follows: in Section 2 we recall some basic facts and we introduce some useful technical results. In Section 3 we present our results. In Section 4 we present same final remarks and we point to further line of investigation.

\textbf{Acknowledgments.} The first author wishes to thank the second author (his Ph.D. advisor) for the patient guidance, encouragement, and advice. The authors owe also special thank to the anonymous referee for the comments that improved the manuscript.

\section{Basic facts}\label{sectionbasicfacts}

\subsection{Apolarity}

In this subsection we briefly recall some facts about apolarity theory, see also \cite{geramita1996inverse} and \cite{iarrobino1999power}.
Let $S=\mathbb{C}[x_0,\ldots,x_n]$ and $T=\mathbb{C}[y_0,\ldots,y_n]$. We make $S$ into a $T$-module via differentiation, that is we think of $y_j =\partial/\partial x_j $. For any form $ F $ of degree $ d $ in $ S $, we define the ideal $ F^\perp\subseteq T $ as follows:
$$F^\perp=\{\partial\in T:\partial F=0\}. $$
Given a homogeneous ideal $ I\subseteq T $ we denote by
$$ \operatorname{HF}(T/I,i)=\dim_kT_i-\dim_kI_i $$
its $ Hilbert\ function $ in degree $ i $. It is well known that for all $i >>0$ the function $ \operatorname{HF}(T/I,i) $ is a polynomial function with rational coefficients, called the $ Hilbert $ $polynomial $ of $ T/I $. We say that an ideal $I\subseteq T $ is $ one\ dimensional $ if the Krull
dimension of $ T/I $ is one, equivalently the Hilbert polynomial of $ T/I $ is some integer constant, say $ s $. The integer $ s $ is then called the $ multiplicity $ of $ T/I $. If, in addition, $ I $ is a radical ideal, then $ I $ is the ideal of a set of $ s $ distinct points in $ \mathbb{P}^n=\mathbb{P}(S_1) $. We will use the
fact that if $ I $ is a one dimensional saturated ideal of multiplicity $ s $, then  $ \operatorname{HF}(T/I,i) $ is always $ \leqslant s $.

The following Lemma, which we will call the Apolarity Lemma, is a consequence of
\cite[Lemma 1.31]{iarrobino1999power}.

\begin{lemma}[Apolarity Lemma]\label{ourapolaritylemma}
A degree d form $F\in S$ can be written as
$$F=\displaystyle\sum_{i=1}^s \alpha_i{L}^d_i,\ {L}_i\in S_1\  pairwise\ linearly\ independent,\ \alpha_i\in\mathbb{C} $$
if and only if there exists $ I\subseteq F^\perp $ such that $ I $ is the ideal of a set of $ s $ distinct points in $ \mathbb{P}(S_1) $.
\end{lemma}

\subsection{Star configurations}

In this subsection we briefly recall some facts about star configuration set of points, see also \cite{starongenerichyper}.

\begin{definition}
Let $ l_1,\ldots, l_r $ be $ r $ linear forms in $ T $ such that any subset of $n+1$ is linearly independent. A star configuration set of points in $\mathbb{P}^n $ is the set of $ \binom{r}{n} $ points obtained by intersecting $ n $ of the hyperplanes $ \{l_i=0\}$ in all possible ways, that is
 $ \mathbb{X}(r) $ is the algebraic variety in $\mathbb{P}^n $ defined by the homogeneous ideal
$$J=\displaystyle\bigcap_{\substack{\tau=\{j_1,\ldots,j_{n} \}\subseteq[r]\\}} (l_{j_{1}},\ldots,l_{j_{n}}).$$
\end{definition}

Star configuration set of points behave like generic points from the point of view of Hilbert functions.

\begin{theorem}\rm{(\cite[Theorem 2.5]{starongenerichyper})}. \label{theoremin"n"}
Let $ \mathbb{X}(r)\subset\mathbb{P}^n $ be a star configuration of points. Then $ \mathbb{X}(r) $ has a generic Hilbert function, that is,
$$\operatorname{HF}( \mathbb{X}(r), t) = \dim \mathbb{C} (S/I({ \mathbb{X}(r)} ))_t = \min\left\lbrace\binom{n+t}{t},\binom{r}{n} \right\rbrace $$
Furthermore, the ideal $ I({\mathbb{X}(r)})  $ is generated by $ \binom{r}{n-1} $ forms of degree $ r-n+1 $.
\end{theorem}

\subsection{A necessary condition} In this subsection we introduce a necessary condition for the existence of star configurations apolar to generic forms.

\begin{proposition}\label{prop1}
Let $ r\geq 3  $, $ d\geq 2 $ and $ n\geq 2 $ be integers. If $ F $ is a generic degree $ d $ form in $n+1$ variables such that there exists a star configuration $\mathbb{X}(r)$ apolar to $F$, then $ \rho(d,r,n)\geq0 $ where,
 $$ \rho(d,r,n)=\binom{r}{n}+nr-\binom{d+n}{d}. $$
\end{proposition}
\begin{proof}
We describe all star configurations $\mathbb{X}(r) $ in $ \mathbb{P}^n $. Let $ \check{\mathbb{P}}^n	 $ be the dual projective space of $ \mathbb{P}^n $ and let $ \ell_i\in\check{\mathbb{P}}^n$ be the corresponding hyperplane to $ 
l_i\in T_1 $. We consider the quasi-projective variety
\begin{center}
$\mathcal{D}_r\subseteq\underbrace{\check{\mathbb{P}}^n\times\ldots\times\check{\mathbb{P}}^n}_{r-times}=(\check{\mathbb{P}}^n)^r,$
\end{center}
where $ (\ell_1,\ldots,\ell_r)\in\mathcal{D}_r$ if and only if no $ n+1 $ of the hyperplanes $ \ell_i $ pass through the same point. Since $ \mathbb{P}^n\cong \mathbb{P}(S_1) $, it follows that any point $ p_i\in\mathbb{X}(r)\subset \mathbb{P}^n$ can be seen as the point $ [{L}_i]\in\mathbb{P}(S_1) $, $ L_i\in S_1 $ and so  $\mathbb{X}(r)=\{[{L}_1],\ldots,[{L}_{\binom{r}{n}}]\} $.  Let us consider the following Veronese map:
\begin{center}
\begin{tabular}{lcr}
$ \nu_d: $	& $ \mathbb{P}(S_1)\cong\mathbb{P}^n $ $ \longrightarrow  $ $ \mathbb{P}^{N_{d,n}}\cong\mathbb{P}(S_d) $,  &\qquad$  N_{d,n}=\binom{d+n}{d}-1. $\\
&$ [{L}_i] $ $ \longmapsto $ $ [{L}_i^d] $ &
\end{tabular}
\end{center}
Let $H$ be the projectivization of the linear span of the set $ \{\nu_d([{L}_1]),\ldots,\nu_d([{L}_{\binom{r}{n}}])\} $, that is, $H=\mathbb{P}(\langle [{L}_1^d],\ldots,[{L}_{\binom{r}{n}}^d]\rangle  )$.  
By Theorem \ref{theoremin"n"} we have that 
\begin{flalign*}
\operatorname{HF}(\mathbb{X}(r),d)=&\dim_{\mathbb{C}} \left(T/I({\mathbb{X}(r)})\right)_d\\
=&\dim_\mathbb{C}\langle[{L}_1^d],\ldots,[{L}_{\binom{r}{n}}^d]\rangle=\begin{cases}
\binom{r}{n}\qquad\quad  \forall\  d\geq r-n+1\\
\binom{d+n}{n}\qquad \forall\  d\leq r-n.
\end{cases}
\end{flalign*}
Therefore, $ \dim H=\min\left\{\binom{r}{n},\binom{d+n}{d}\right\}-1$. Define
$$\Psi:\mathcal{D}_r\longrightarrow \operatorname{Gr}\left(\mathbb{P}^{\dim H},\mathbb{P}^{N_{d,n}}\right), $$
which maps $ (\ell_1,\ldots,\ell_r) $ to $ \langle [{L}_1^d],\ldots,[{L}_{\binom{r}{n}}^d]\rangle  $. For a generic point $ [F]\in H $, we have that \[
F=\alpha_1{L}_1^d+\ldots+\alpha_{\binom{r}{n}}{L}_{\binom{r}{n}}^d
\]
and we define the following incidence correspondence:
$$\Sigma(d,r,n)=\{((\ell_1,\ldots,\ell_r),[F]):[F]\in \langle [{L}_1^d],\ldots,[{L}_{\binom{r}{n}}^d]\rangle \}\subseteq \mathcal{D}_r\times\mathbb{P}^{N_{d,n}}.  $$
We also consider the natural projection maps
\[
\pi_1:\Sigma(d,r,n)\longrightarrow \mathcal{D}_r \mbox{ and } \pi_2:\Sigma(d,r,n)\longrightarrow \mathbb{P}^{N_{d,n}}.
\]
Using a standard fiber dimension argument for a generic $ (\ell_1,\ldots,\ell_r)\in\mathcal{D}_r$,  follows that
$$
\dim(\Sigma(d,r,n))\leq\dim\pi^{-1}_1\left((\ell_1,\ldots,\ell_r)\right)+\dim\mathcal{D}_r=\dim H +nr.
$$
The map $ \pi_2 $ is dominant if and only if the generic degree $d$ form $ F $ in $n+1$ variable has an apolar $\mathbb{X}(r)$.
The map $ \pi_2 $ is dominant only if $ \dim(\Sigma(d,r,n))-\dim(\mathbb{P}^{N_{d,n}})\geq0 $ and this implies
\begin{align}\label{eqroh}
&\dim H +nr-N_{d,n}\geq 0.
\end{align}
It follows that for $ d\geq r-n+1 $, $$\rho(d,r,n)=\binom{r}{n}+nr-\binom{d+n}{d}\geq 0.$$
Note that for $ d\leq r-n $ 
$$\rho(d,r,n)\geq\binom{r}{n}+\binom{d+n}{d}-\binom{d+n}{d}> 0.$$ 

\end{proof}
It is useful to specialize the necessary condition in the case $n=2$.
\begin{corollary}\label{corn=2}
Consider the previous proposition. If $ n=2 $, then
$$\rho(d,r,2)=\dfrac{1}{2}(r(r-1)+4r-(d+2)(d+1)). $$
\end{corollary}

%Another necessary condition is given by the rank of the generic degree $d$ form in $n+1$ variables $F$, that is we must have that $\mathrm{rk}(F)\leq|\mathbb{X}(r)|$; see \cite{alexander1995polynomial} for the value of $\mathrm{rk}(F)$ for the generic form.

\subsection{A computational approach} \label{RJ}

It is possible to decide whether the generic degree $d$ form in $n+1$ variables has an apolar star configuration $\mathbb{X}(r)$ using a computational approach. However, the computational complexity is prohibitive, and this approach does effectively produce an answer only for small values of $d,r$, and $n$.

Let us recall the natural projection map $ \pi_2:\Sigma(d,r,n)\longrightarrow \mathbb{P}^{N_{d,n}} $ from Proposition \ref{prop1}. Let $ d\geq 3 $ and $ n\geq 2 $ be integers. The closure of the image of $ \pi_2 $ is the closure of the union of the linear spans of all possible $ \mathbb{X}(r)\subset\mathbb{P}^n $, and we denote it by
$$\mathcal{U}(d,r,n):=\overline{\operatorname{Im}\pi_2}. $$
We only consider  $ (d,r,n) $ such that $ \rho(d,r,n)\geq0 $ because of Proposition \ref{prop1}. To compute $\dim\mathcal{U}(d,r,n)$, it is enough to find the dimension of the tangent space to $ \operatorname{Im}\pi_2$ at a generic point $ p$.
$$\overline{\operatorname{Im}\pi_2}=\overline{\displaystyle\bigcup_{\mathbb{X}(r)\subset\mathbb{P}^n}\langle [{L}_1^d],\ldots,[{L}_{\binom{r}{n}}^d]\rangle }.$$
In order to compute algorithmically the dimension of the tangent space, we proceed as follows.

We construct $ r $ linear forms $ {l}_1,\ldots,{l}_r $ using $ (n+1)r $ variables,
$${l}_1=a_{0,1}y_0+a_{1,1}y_1+\cdots +a_{n,1}y_n,\ldots, {l}_r=a_{0,r}y_0+a_{1,r}y_1+\cdots +a_{n,r}y_n.$$
Let $ \mathbb{X}(r)=\{p_1,p_2,\dots,p_{\binom{r}{n}} \} $ be the set of points that is obtained by constructing the star configuration  $ \mathbb{X}(r) $ using $ {l}_1,\dots,{l}_r $. Note that any $p_i=[b_{0,i},b_{1,i},\cdots,b_{n,i}] $ for $ i=1,\ldots,\binom{r}{n} $ is such that 
$$b_{j,i}=f_{j,i}(a_{0,1},\ldots,a_{0,r};\ldots;\stackon[-8pt]{$ a_{j,1},\ldots,a_{j,r} $}{\vstretch{1.5}{\hstretch{2.4}{\widehat{\phantom{\;\;\;\;\;\;\;\;}}}}};\ldots;a_{n,1},\ldots,a_{n,r}),\ \ j=0,\ldots,n$$
where $ f_{j,i} $ is a polynomial and $ \stackon[-8pt]{$a_{j,1},\ldots,a_{j,r}$}{\vstretch{1.5}{\hstretch{2.4}{\widehat{\phantom{\;\;\;\;\;\;\;\;}}}}} $ means that the variables $a_{j,1},\ldots,a_{j,r}$ do not appear in $ b_{j,i} $. For a pair $ ((\ell_1,\ldots,\ell_r),[F])\in\Sigma(d,r,n) $, $ F $ is a form of degree $ d $ with $ m $ variables where $m=(n+1)r+\binom{r}{n}  $ such that 
$$
F=\alpha_1{L}_1^d+\cdots+\alpha_{\binom{r}{n}}{L}_{\binom{r}{n}}^d;\qquad {L}_i=b_{0,i}x_0+b_{1,i}x_1+\cdots+b_{n,i}x_n. $$
Let $ g_i:=\operatorname{coeff}_{m_i}(F) $, where $ m_i $ is the $ i$-th element of the standard monomial basis of $ S_d $ respect to the $ lexicographic $ order, for $ i=1,\ldots,\binom{d+n}{d} $. We define  the map
$$ \Gamma:\mathbb{A}^{m}\longrightarrow \mathbb{A}^{N_{d,n}+1}$$
which maps every $ F $  to  $(g_1,g_1,\ldots,g_{N_{d,n}+1}) $. Then we compute the $rank $ of the Jacobian matrix $m\times(N_{d,n}+1)  $ of the map evaluated at a generic point $ p$.
Recalling that
$$\dim\mathcal{U}(d,r,n)=\dim\overline{\operatorname{Im}\pi_2}=\operatorname{rank}(\operatorname{Jac}\Gamma)_p-1,$$
we can use this computational approach to address our question, namely

\begin{lemma}\label{jaclemma}
The generic degree $d$ form in $n+1$ variables has an apolar $\mathbb{X}(r)$ if and only if
\[
\operatorname{rank}(\operatorname{Jac}\Gamma)_p={n+d \choose d}
\]
for some choice of the parameters $p$.
\end{lemma}

%%%%%%%%%%%%%%%%%%%%%%%%%%%%%%%%%%%%%%%%%%%%%%%%%%%%%%%%%%%%%%%%%%%%%%%%%%%%%%%

\section{Main results}

In this section we present our main results about the question: for what 3-ples $(d,r,n)$ does the generic degree $d$ form in $n+1$ variables have an apolar star configuration $\mathbb{X}(r)$?

\begin{lemma}\label{sufficient_condition_from_ideal}
If $r\geq d+n$, then the generic degree $d$ form in $n+1$ variables has an apolar star configuration $\mathbb{X}(r)$.
\end{lemma}
\begin{proof}
If $F$ is a generic degree $d$ form in $n+1$ variables, then $(F^\perp)_j=T_j$ for $j\geq d+1$. By  Theorem \ref{theoremin"n"} the ideal of a star configuration $\mathbb{X}(r)$ starts in degree $r-n+1$ and the conclusion follows.
\end{proof}

In the case of quadrics, i.e. $d=2$, we can immediately give a complete answer:

\begin{lemma}
The generic quadratic form in $n+1$ variables has an apolar star configuration $\mathbb{X}(r)$ if and only if $r\geq n+1$.
\end{lemma}
\begin{proof}
Using Lemma \ref{sufficient_condition_from_ideal} we only need to consider $r\leq n+1$. If $r<n+1$, there is no star configuration  $\mathbb{X}(r)$ and thus the result follows for $r<n+1$. If $r=n+1$, the result follows since $\mathbb{X}(n+1)$ consists of $n+1$ points in general position and the rank of the generic quadratic form is $n+1$.
\end{proof}

Since the degree two case is completely solved, we now consider the $d\geq 3$. We first consider the case $n\geq 6$ for which we have a very uniform solution.

\begin{theorem}
Let $ d\geq 3 $ and $ n\geq 6 $ be integers. If $ r< d+n $, then there is no star configuration $ \mathbb{X}(r) $ apolar to a generic form of degree $ d $. If $ r\geq d+n $, then there exists a star configuration $ \mathbb{X}(r) $ apolar to a generic form of degree $ d $.
\end{theorem}

\begin{proof}
Because of Lemma \ref{sufficient_condition_from_ideal} we only need to consider $r<d+n$.

If  $ r< d+n $, we claim that $ \rho(d,r,n)<0 $, and then, by Proposition \ref{prop1}, there is no star configuration $ \mathbb{X}(r) $ apolar to a generic form of degree $ d $.

\begin{proof} [Claim]

For $ r\leq d+n-1 $, we have that $ \binom{r}{n}\leq\binom{d+n-1}{n} $ and $ nr\leq n(d+n-1) $. Thus,
\begin{align*}
\rho(d,r,n)&={r\choose n}+nr-{d+n\choose d} \\
&\leq{d+n-1\choose n}+n( d+n-1) -{d+n\choose d}\\
&=n( d+n-1)-\binom{d + n - 1}{d}\\
&= n( d+n-1)-\frac{1}{d!} (d + n - 1) \cdots (n + 1)n\\
&= n( d+n-1)-\frac{1}{d(d - 1)}\binom{d + n - 2}{d - 2}n(d + n - 1)\\
&= n( d+n-1)\left(1-\frac{1}{d(d - 1)}\binom{d + n - 2}{d - 2} \right). \\
\end{align*}
Since $ n( d+n-1)>0 $, then it suffices to prove that $
\binom{d + n - 2}{d - 2} > d(d - 1),
$
which is true because
\begin{align*}
\binom{d + n - 2}{d - 2} &> \binom{d + 5 - 2}{d - 2} = \frac{1}{5!} (d + 3)(d + 2)(d + 1)d(d - 1)\\
&\geqslant \frac{1}{5!} (3 + 3)(3 + 2)(3 + 1)d(d - 1) = d(d - 1).
\end{align*}
Hence, for $ r< d+n $ there is no $ \mathbb{X}(r) $ apolar to the generic form of degree $ d $. The claim is now proved.
\end{proof}
The proof is now completed.
\end{proof}

We now consider the cases $n=3,4,5$.

\begin{theorem}\label{n=3,4,5}
Let $ d\geq 3 $ be an integer and $ n=3,4,5 $. Let $ F $ be a generic form of degree $ d $ in $n+1$ variables. If $ r\geq d+n $, then there exists a star configuration $ \mathbb{X}(r) $  apolar to $ F $. If $ r< d+n $, then there does not exist a star configuration $ \mathbb{X}(r) $  apolar to $ F $ unless the 3-ple $ (d,r,n) $ is one of the following cases in which we have existence: $$(3,5,3),(4,6,3),(5,7,3),(3,6,4), \text{or}\  (3,7,5).$$
\end{theorem}
\begin{proof}
By Lemma \ref{sufficient_condition_from_ideal}, we conclude that for any $ r\geq d+n $ there exists a star configuration  $ \mathbb{X}(r) $ apolar to $ F $. Now, assume that $ r<d+n $ and consider the following cases:\\
\textbf{(a)} If $ r=d+n-1 $, then
\begin{align*}
\rho(d,d+n-1,n)&={d+n-1\choose n}+n( d+n-1) -{d+n\choose d}\\
&=n( d+n-1)-\binom{d + n - 1}{d},
\end{align*}
and we have the following cases:
\begin{itemize}
	\item[\textbf{(1)}] case $ n=3 $
	$$\rho(d,d+2,3)=3( d+2)-\binom{d + 2}{d}=(d+2)(5-d)/2.$$
	Therefore, for $ d\geq 3 $ we have $ (d+2)(5-d)/2<0 $ unless $ d=3,4,5 $.
	\item[\textbf{(2)}] case $ n=4 $
\begin{align*}
\rho(d,d+3,4)=4( d+3)-\binom{d +3}{d}=\left( (d+6)(3-d)+4\right)(d+3)/6.
\end{align*}
	Hence, for all $ d\geq3 $,  $ \rho(d,d+3,4)<0 $  unless $ d=3 $.
	\item[\textbf{(3)}] case $ n=5 $
	$$\rho(d,d+4,5)=5( d+4)-\binom{d +4}{d}=(d+4)(3-d)(d^2+9d+38)/24.$$
	We conclude that $ \rho(d,d+4,5)<0 $ for all $ d\geq3 $ unless $ d=3 $.
\end{itemize}
\textbf{(b)} If $ r\leq d+n-2 $, then we have $ \binom{r}{n}\leq\binom{d+n-2}{n} $ and $ nr\leq n(d+n-2) $. Hence,
\begin{align*}
\rho(d,r,n)={r\choose n}+nr-{d+n\choose d} \leq{d+n-2\choose n}+n(d+n-2)-{d+n\choose d}.
\end{align*}
As in part \textbf{(a)}, we consider the following cases:
\begin{itemize}
\item[\textbf{(1)}] case $ n=3 $
$$ \rho(d,r,3)\leq{d+1\choose 3}+3(d+1)-{d+3\choose d}=-(d+1)(d-2).$$ It is obvious to see that $ -(d+1)(d-2)<0 $ for any  $d\geq3 $ .
\item[\textbf{(2)}] case $ n=4 $
\begin{align*}
 \rho(d,r,4)&\leq{d+2\choose 4}+4(d+2)-{d+4\choose d}
 =-(d+2)(2d^2+5d-21)/6\\
 &\leq-2(d+2).
\end{align*}
So, for all $ d\geq3 $ we have that $ -2(d+2)<0 $.
\item[\textbf{(3)}] case $ n=5 $
\begin{align*}
\rho(d,r,5)&\leq{d+3\choose 5}+5(d+3)-{d+5\choose d}=(d+3)(d^3+5d^2+8d-56)/{6}\\
&\leq-10(d+3)/3.
\end{align*}
 Therefore, $ -10(d+3)/3<0 $ for all $ d\geq3 $.
\end{itemize}
Hence, for $ r< d+n $ the necessary condition is not satisfied, $ \rho(d,r,n)<0 $, except for five 3-ples which have appeared in \textbf{(a)(1)}, \textbf{(a)(2)}, and \textbf{(a)(3)}.
 So, to complete the proof we only need to prove that in the above five cases we have existence. By the strategy in Subsection \ref{RJ}, if we show that  $ \dim\mathcal{U}(d,r,n)=N_{d,n} $ for the above cases, then the proof is completed and this is done computationally using \verb|Macaulay2|, \cite{M2}.
\end{proof}

We now conclude this section with the $n=2$ case in which we have a complete solution for all 3-ples not of the form $(d,d+1,2)$.

\begin{proposition}
Let $F$ be a generic degree $d$ form in three variables. If $r\geq d+2$, then there exists a star configuration $\mathbb{X}(r)$ apolar to $F$.  If $r\leq d$, then there does not exists a star configuration $\mathbb{X}(r)$ apolar to $F$.
\end{proposition}
\begin{proof}
The case $r\geq d+2$ is proved using Lemma \ref{sufficient_condition_from_ideal}. The case  $ r-d\leq 0 $, is proved using Proposition \ref{prop1}. In fact, by Corollary \ref{corn=2}, we have
\[
\rho(d,r,2)=\dfrac{1}{2}\left(r(r-1)+4r-(d+2)(d+1)\right)=\dfrac{1}{2}(r-d)(3+r+d)-1<0,
\]
and hence we conclude that there is no star configuration $\mathbb{X}(r) $ apolar to $F$.
\end{proof}

The cases $(d,d+1,2)$ can be treated computationally for small values of $d$ using Lemma \ref{jaclemma} showing that, for $d\leq 13$, there exists a star configuration $\mathbb{X}(d+1)$ apolar to the generic degree $d$ ternary form. This leads to the following conjecture:

\begin{conjecture}\label{conj}
Let $ d\geq 3 $ be an integer. For a generic ternary degree $ d $ form  $ F  $ there exists a star configuration $ \mathbb{X}(d+1) $ apolar to it.
\end{conjecture}

\begin{remark}\label{remd=3}
One possible theoretical approach to Conjecture \ref{conj}, successful for the 3-ple $ (3,4,2) $, is the following. It is easy to see that the generic ternary cubic $F$ has a an apolar set of four points which are the complete intersection of two (reducible) conics, that is
\[
F^\perp\supset (l_1l_2,l_3l_4),
\]
thus $ F^\perp\supset I=(l_2l_3l_4,l_1l_3l_4,l_1l_2l_4,l_1l_2l_3) $ and $I$ is the ideal of a star configuration $\mathbb{X}(4)$. Hence the conjecture is proved for $d=3$.
\end{remark}

\begin{remark}

One possible computational approach to Conjecture \ref{conj}, successful for $d\leq 13$, uses Lemma \ref{jaclemma}. 

\end{remark}

%%%%%%%%%%%%%%%%%%%%%%%%%%%%%%%%%%%%%%%%%%%%%%%%%%%%%%%%%%%%%%%%%%%%%%%%%%%%%%%

\section{Final remarks}

Our main results are {\em generic} results, that is they hold for the generic degree $d$ form in $n+1$ variables, and not for {\em any} such a form. However, in some cases, we can show that our results hold for any form. In what follows, we deal with ternary cubics, that is $n=2$ and $d=3$.

\begin{lemma}\label{lemcusps}
Any ternary cuspidal cubic is projectively equivalent to $V(  x_0^3-x_1^2x_2) $.
\end{lemma}
\begin{proof}
See, \cite[Lemma 15.5]{Gibson:1998:EGA:298604}.
\end{proof}

\begin{proposition}\label{propnongen}
Any ternary cuspidal cubic has an apolar star configuration $ \mathbb{X}(4) $.
\end{proposition}
\begin{proof}
By Lemma \ref{lemcusps}, it is enough to show that the normal form ${C}= x_0^3-x_1^2x_2 $ has an apolar star configuration $\mathbb{X}(4)$.
By \cite{carlini2012thesolution}, we know that $ \operatorname{rk}(x_0^3-x_1^2x_2 )=4 $.  Computing we get that
$${C}^\perp=(y_2^{2},y_0 y_2,y_0 y_1,y_1^{3},y_0^{3}+3 y_1^{2} y_2).$$
Using \verb|Macaulay2| we construct linear forms
\[
{l}_1=y_0,\ {l}_2= y_1,\ {l}_3= y_1-y_2,\ {l}_4=y_0+y_1+y_2,
\]
defining a star configuration $\mathbb{X}(4)$ apolar to ${C} $. This completes the proof.
\end{proof}

\begin{lemma}\label{cone+line}
Any rank five ternary cubic is projectively equivalent to $V( x_0(x_2^2+x_0x_1)) $.
\end{lemma}
\begin{proof}
See \cite[Lemma 15.6]{Gibson:1998:EGA:298604}.
\end{proof}

\begin{proposition}\label{probline}
There exists an apolar star configuration $ \mathbb{X}(4) $ for any ternary cubic of rank five (conic plus tangent line).
\end{proposition}
\begin{proof}
Using Lemma \ref{cone+line}, we only need to find an apolar star configuration $ \mathbb{X}(4) $ for the normal form of conic plus tangent type ${G}= x_0(x_2^2+x_0x_1) $. Computing we get
\[
{G}^\perp=({y}_{1} {y}_{2},{y}_{1}^{2},{y}_{0} {y}_{1}-{y}_{2}^{2},{y}_{0}^{2} {y}_{2},{y}_{0}^{3}).
\]
Using \verb|Macaulay2| we construct linear forms
\[
l_1={y}_{0}+(47/132) {y}_{1}-3 {y}_{2},\  l_2=4 {y}_{0}-(20/3) {y}_{1}-10 {y}_{2},
\]
\[
l_3=2 {y}_{0}+(862/33) {y}_{1}+7 {y}_{2},\  l_4=11 {y}_{0}-(421/12) {y}_{1}+6 {y}_{2}
\]
defining a star configuration $ \mathbb{X}(4) $ apolar to ${G} $. This completes the proof.
\end{proof}

In conclusion, we have the following result.

\begin{theorem}
Any ternary cubic form has an apolar star configuration $ \mathbb{X}(4) $.
\end{theorem}
\begin{proof}
It is easy to see that the ternary cubics of rank one (triple line), rank two (three concurrent lines), and rank three (double line + line and smooth) have an apolar star configurations $ \mathbb{X}(r) $ for $ r\geq 3 $. For the ternary cubics of rank four including three non-concurrent lines, line + conic (meeting transversally), nodal, and  general smooth, see Remark \ref{remd=3}. For the case of ternary  cuspidal cubic (rank four) we refer to either Remark \ref{remd=3} or Proposition \ref{propnongen}. The result for the ternary cubic of rank five (line + tangent conic) follows from Proposition \ref{probline}. The proof is now completed.
\end{proof}
%%%%%%%%%%%%%%%%%%%%%%%%%%%%%%%%%%%%%%%%%%%%%%%%%%%%%%%%%%%%%%%%%%%%%%%%%%%%%%%

\end{document}